\documentclass[11pt]{amsart}
\pdfoutput=1 

\usepackage{amsmath}
\usepackage{textcomp} 
\usepackage[font=footnotesize]{subcaption}

\usepackage{scalerel}
\usepackage{stackengine,wasysym} 
\usepackage{mathdots} 
\newlength\bshft 
\def\fakebold#1{\ThisStyle{\ooalign{$\SavedStyle#1$\cr%
  \kern-\bshft$\SavedStyle#1$\cr%
  \kern\bshft$\SavedStyle#1$}}}

\usepackage[usenames,dvipsnames]{xcolor}
\usepackage[mathscr]{euscript}
\usepackage{array, amsmath, amsfonts, amssymb, amsthm, geometry, graphics, graphicx, multicol, multirow, enumerate, float, floatflt, fullpage, tikz, outlines, url}
\usepackage[utf8]{inputenc}
\usetikzlibrary{shapes.multipart, matrix, positioning, arrows}
\newcolumntype{P}[1]{>{\centering\arraybackslash}p{#1}} 

\usepackage{tikz}
\usepackage{diagbox} 
\usetikzlibrary{shapes,decorations}

\theoremstyle{definition} 

\newtheorem{defn}{Definition}
\newtheorem{thm}{Theorem}[section]
\newtheorem{cor}[thm]{Corollary}
\newtheorem{prop}[thm]{Proposition}

\newtheorem{rem}{Remark}
\newtheorem{exa}{Example}

\newcommand{\Z}{\mathbb{Z}}

\newcommand{\A}{\mathcal{A}}
\newcommand{\al}{\alpha}
\newcommand{\be}{\beta}

\newcommand{\la}{\lambda}
\newcommand{\si}{\sigma}

\newcommand{\gr}{\textnormal{gr}}
\newcommand{\rk}{\textnormal{rk}}
\newcommand{\tor}{\textnormal{tor$_2$}}
\newcommand{\rktor}{\textnormal{rk tor$_2$}}

\newcommand\Widetilde[1]{\ThisStyle{%
  \setbox0=\hbox{$\SavedStyle#1$}%
  \stackengine{-.1\LMpt}{$\SavedStyle#1$}{%
    \stretchto{\scaleto{\SavedStyle\mkern.2mu\AC}{.5150\wd0}}{.6\ht0}
  }{O}{c}{F}{T}{S}%
}}

\title{Extremal Khovanov homology and the girth of a knot}

\author[R. Sazdanovi\'{c}]{Radmila Sazdanovi\'{c}}
\thanks{RS partially supported by the Simons Foundation Collaboration Grant 318086  and NSF Grant DMS 1854705.}
\address{Department of Mathematics, North Carolina State University, Raleigh, NC 27695}
\email{rsazdan@ncsu.edu}

\author{Daniel Scofield}
\address{Department of Mathematics\\
Francis Marion University\\
Florence, SC}
\email{daniel.scofield@fmarion.edu}

\begin{document}

\maketitle

\begin{abstract} We utilize relations between Khovanov and chromatic graph homology to determine extreme Khovanov groups
and corresponding coefficients of the Jones polynomial. The extent to which chromatic homology and chromatic polynomial can be used to compute integral Khovanov homology of a link depends on the maximal girth of its all-positive graphs. In this paper we also define the girth of a link, discuss relations to other knot invariants, and the possible values for girth. Analyzing girth leads to a description of possible all-A state graphs of any given link; e.g., if a link has a diagram such that the girth of the corresponding all-A graph is equal to $\ell>2$, than the girth of the link is equal to $\ell.$
    
\end{abstract}

\section{Introduction}

Khovanov homology \cite{Khov1} is a bigraded homology theory which is an invariant of knots and links, categorifying the Jones polynomial. In general, the structure of Khovanov homology and the types of torsion which occur may vary widely \cite{DBN, MPS, Stosic1}. For certain links, there is a partial isomorphism between the extreme gradings of Khovanov homology and chromatic graph homology, a categorification of the chromatic polynomial for graphs \cite{AP,PS}. The isomorphism between these two theories describes a part of Khovanov homology that is supported on two diagonals and has only $\Z_2$ torsion, similar to the Khovanov homology of an alternating link. Moreover, this correspondence allows us to describe ranks of groups in Khovanov homology in terms of combinatorial information from a diagram, or a graph associated to the diagram.
		
Khovanov homology of alternating knots is determined by the Jones polynomial and the signature of a knot and, similarly chromatic graph homology over the algebra $\A_2 = \Z[x]/(x^2)$ is determined by the chromatic polynomial \cite{LS}. This approach enables us to determine some extremal Khovanov homology groups based on combinatorial results about the chromatic polynomial of a graph which determines its chromatic homology. The following theorem illustrates the type of the results we obtain.

\noindent{\bf Theorem~\ref{rankKhovanovgirth}}\textit{	Let $D$ be a diagram of a link $L$ such  that the all-positive graph of $D$ has girth $\ell$ and satisfies the conditions of Theorem \ref{rankHGRgirth}.  For $0 < i < \ell$, the ranks of Khovanov homology  groups of $L$ are given by:
	\begin{equation*}
	\rk Kh^{i-c_-(D),N+2i}(L) = \left(\displaystyle\sum_{\substack{r \ge 0,\\ 0\le k = i-2r \le i}} \binom{p_1-2+k}{k} \right)-n_{i+1} +(-1)^{i+1} \delta^b
	\end{equation*}
	where $p_1$ is the cyclomatic number of the graph, $n_{i+1}$ is the number of $(i+1)$-cycles, and $\delta^b$ measures bipartiteness.}

The applicability of our results depends on a quantity defined in Section \ref{Girth1} that we call the girth of a link. We find upper bounds for the value of this invariant based on Khovanov homology and the Jones polynomial. We prove results on the girth of connected sums and of alternating knots, describing another upper bound in terms of crossing number and signature.

Analyzing girth of a link leads to a somewhat surprising characterization of the types of graphs that can be obtained from a homogeneous resolution of diagrams of a given knot (all-positive or all-A state graph)

\noindent{\bf Theorem~\ref{girthPossibleKhovanov}}\textit{	Let $D$ be a diagram of a non-trivial link $L$ such  that the all-positive graph of $D$ has girth $\ell.$ Then either the girth of a link equals $\gr(L)=\ell$ or $\ell \in \{1,2\}.$
}

As a consequence we get that if a link has a diagram such that the girth of the corresponding graph is equal to some $\ell>2$, than the girth of the link is equal to $\ell$, see Corollary \ref{girth3more}. In other words, this is saying that if a link $L$ has girth greater than two, all of the corresponding all-A graphs have girth equal to $\gr(L)$, one or two.

\section*{Acknowledgements}
We are grateful to Adam Lowrance for many ideas and useful discussions. RS was partially supported by the Simons Foundation Collaboration Grant 318086  and NSF Grant DMS 1854705.

\section{Background}

\subsection{Jones polynomial}

Let $D$ be a diagram of link $L$. Each crossing of $D$ can be resolved with a positive or negative resolution as shown below. The positive and negative resolutions are sometimes referred to as the A and B resolutions, respectively (see e.g. \cite{DasLin}).

 \begin{figure}[h]
	\centering
	\includegraphics[scale = 0.6]{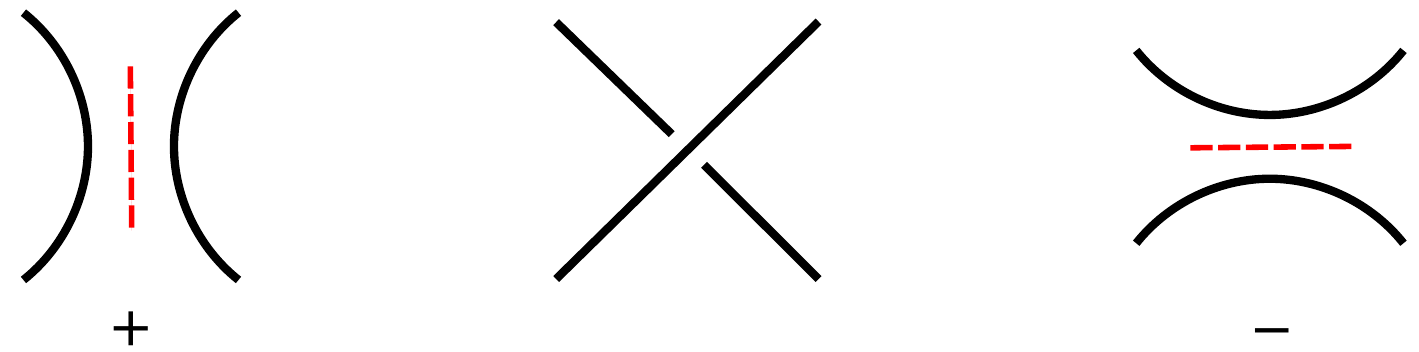}
	\caption{Positive and negative resolutions at a crossing.}
\end{figure}

 The resolution of all crossings in a diagram $D$ produces a collection of disjoint circles known as Kauffman states. From any Kauffman state $s$, we may construct a graph whose vertices correspond to the circles of $s$, and whose edges connect circles whose arcs were obtained by smoothing a single crossing. 
 The Kauffman state $s_+(D)$ is obtained by applying the positive resolution to every crossing in $D$, and we denote the graph obtained from this state by $G_+(D)$ (known as the all-positive or all-A state graph of $D$). Similarly, we define a state $s_-(D)$ with all negative resolutions along with its graph $G_-(D)$.
 
 \begin{figure}[h]
	\centering
	\includegraphics[scale = 0.7]{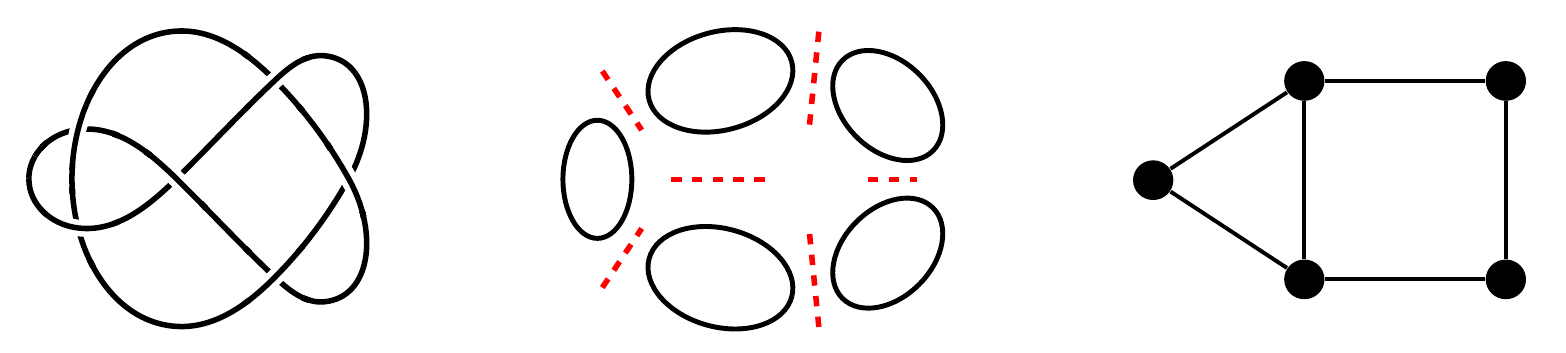}
	\caption{The Kauffman state $s_+(D)$ and the graph $G_+(D)$.}
\end{figure}

We give a definition of the Jones polynomial using Kauffman states as in \cite{SaSco}.

\begin{defn} \label{JonesDef}
Let $L$ be a link and $D$ a diagram of $L$ with $c_+$ positive crossings and $c_-$ negative crossings. The unnormalized Jones polynomial of $L$ is given by:
\begin{equation*}
	\hat{J}_L(q) = (-1)^{c_{-}}q^{c_{+}-2c_{-}}\sum_{i=0}^{c_++c_-} (-1)^i \sum_{\{s~:~n_-(s) = i\}} q^{i}(q+q^{-1})^{|s|}
\end{equation*}

	where $s$ is a Kauffman state of $D$ with $n_{-}(s)$ negative smoothings and $|s|$ connected components.

The normalized version of the Jones polynomial is 
\begin{equation*}
J_L(q) = \hat{J}_L(q)/(q+q^{-1})
\end{equation*}
where $q+q^{-1}$ represents evaluation on the unknot, $\hat{J}_{\Circle}(q) = q+q^{-1}$. 
\end{defn}

Next we introduce some notation that will be useful when discussing graphs.

\begin{defn}\label{p1}
The cyclomatic number $p_1(G)$ of a connected graph $G$ with $v$ vertices and $E$ edges is equal to $p_1(G) = E-v+1$. For planar graphs such as $G_+(D)$, $G_-(D)$ $p_1$ is equal to the number of bounded faces of the graph.
\end{defn}

\begin{defn}[\cite{DasLin}, \cite{LowSpy}]\label{DLreducedgraph}
Let $D$ be a knot diagram with corresponding all-positive graph $G=G_+(D)$. The simplification $G'$ of $G$ is the graph obtained by deleting any loops
in $G$ and replacing each set of multiple edges with a single edge.

Define $\mu$ to be the number of edges in $G'$ which correspond to multiple edges in $G$.
\end{defn}

 We consider the normalized version of the Jones polynomial and denote the coefficients as follows:
\begin{equation} \label{JonesCoef}
    J_L(q) = \be_0q^{C} + \be_1q^{C+2} + \be_2q^{C+4} + \be_3q^{C+6}+ \ldots + \be_iq^{C+2i} + \ldots
\end{equation}

where $C$, the minimal degree of $J_L(q)$, depends on the link $L$.

For a reduced alternating knot, Dasbach and Lin \cite{DasLin} showed that the first three coefficients of the normalized Jones polynomial may be stated in terms of the all-positive graph $G_{+}(D)$. This result is restated in Theorem \ref{DLJones}.

\begin{thm}[\cite{DasLin}] \label{DLJones}
	Let $K$ be a knot with reduced alternating diagram $D$. Let $p_1$ and $t_1$ be the cyclomatic number and the number of triangles in $G_+(D)'$, and let $\mu$ be defined as above. Then the first three coefficients of $J_K(q)$ (up to an overall change in sign) are  $\be_0 = 1, \be_1 = -p_1,$ and $\be_2 = \binom{p_1+1}{2} + \mu - t_1$. 
\end{thm}

The lowest-degree terms of the Jones polynomial are often referred to as the ``tail," while the highest-degree terms are referred to as the ``head." Note that if the all-positive graph obtained from $D$ is replaced by the all-negative graph in Theorem \ref{DLJones}, a similar result applies to the three extremal coefficients in the head of the Jones polynomial.

\subsection{Chromatic polynomial}

We now define the chromatic polynomial of a graph. Let $G$ be a finite, undirected graph with vertex set $V(G)$ and edge set $E(G)$. We will often denote the cardinalities of these sets by $v = |V(G)|$ and $E = |E(G)|$. If $G$ has an edge between vertices $x, y \in V(G)$, we write the corresponding element in $E(G)$ as $\{x,y\}$.

\begin{defn}[\cite{DKT}] \label{ChromDef}
	A mapping $f:V(G) \to \{1, \ldots, \la\}$ is called a $\la$-coloring of $G$ if for any pair of vertices $x,y \in V(G)$ such that $\{x,y\} \in E(G)$, $f(x) \neq f(y)$. The chromatic polynomial of the graph $G$, denoted $P_G(\la)$, is equal to the number of distinct $\la$-colorings of $G$.
\end{defn}

For any graph $G$, the degree of $P_G(\la)$ is equal to $v$. We will represent the terms of the polynomial as follows:
\begin{equation} \label{ChromCoef}
    P_G(\la) = c_v\la^v + c_{v-1}\la^{v-1} + c_{v-2} \la^{v-2} + \ldots + c_{v-i} \la^{v-i} + \ldots+ c_{1} \la
\end{equation}

The first few coefficients of $P_G(\la)$ can be described in terms of cycles and subgraphs found in $G$.

\begin{defn}
	The girth of a graph $G$, denoted $\ell(G)$, is the number of edges in the shortest cycle in $G$. 
\end{defn}

\begin{defn}
	Let $H$ be a subgraph of graph $G$. We say $H$ is an induced subgraph if for every $\{x,y\} \in E(G)$ with $x, y \in V(H)$, the edge $\{x,y\}$ is in $E(H)$.
\end{defn}

We adopt the convention that the girth of a tree is zero, but it is worth noting that there are different conventions considering girth of a tree to be infinite \cite{Boll1, Diestel}.

\begin{thm}\cite{Meredith1} \label{Meredith}
	If $G$ is a graph with girth $\ell>2$ and $n_{\ell}$ cycles of length $\ell$, then the first $\ell$ coefficients of the chromatic polynomial $P_G(\la)$ are:
	$$c_{v-i} = \begin{cases}
	(-1)^i \displaystyle\binom{E}{i} & 0 \le i < \ell - 1\\
	(-1)^{\ell-1} \left( \displaystyle\binom{E}{\ell-1} - n_{\ell}\right) & i = \ell - 1\\
	\end{cases}$$
\end{thm}

\begin{rem}
	The statement of this result in \cite[Theorem 2]{Meredith1} is not explicitly restricted to graphs with $\ell>2$. In the case $i = \ell-1$, the proof contains an assumption that the number of cycle-containing subgraphs with $v-1$ connected components and $t$ edges is zero for $t > 2$; this is not true for graphs with edge multiplicities greater or equal to 3.
\end{rem}

\begin{figure}[h!]
	\centering
	\begin{subfigure}[b]{0.15\textwidth}
		\centering
		\includegraphics[width=0.56\linewidth]{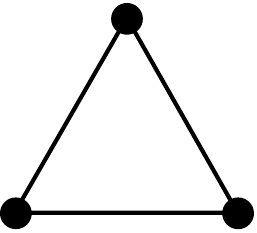}
		\\
		$T_1$
	\end{subfigure}
	\begin{subfigure}[b]{0.15\textwidth}
		\centering
		\includegraphics[width=0.56\linewidth]{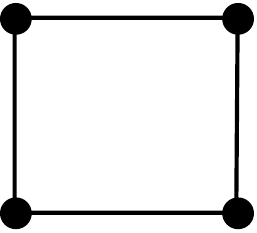}
		\\
		$T_2$
	\end{subfigure}
	\begin{subfigure}[b]{0.15\textwidth}
		\centering
		\includegraphics[width=0.56\linewidth]{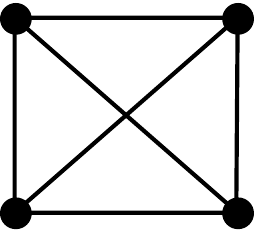}
		\\
		$T_3$
	\end{subfigure}
	\begin{subfigure}[b]{0.15\textwidth}
		\centering
		\includegraphics[width=0.56\linewidth]{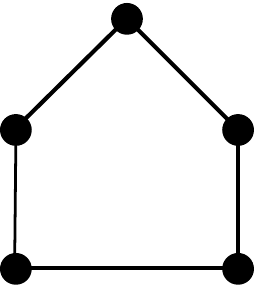}
		\\
		$T_4$
	\end{subfigure}
	\begin{subfigure}[b]{0.15\textwidth}
		\centering
		\includegraphics[width=0.56\linewidth]{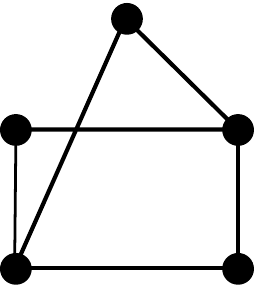}
		\\
		$T_5$
	\end{subfigure}

	\vspace{5mm}
	
	\begin{subfigure}[b]{0.15\textwidth}
		\centering
		\includegraphics[width=0.56\linewidth]{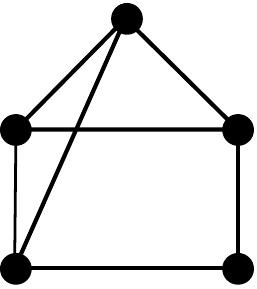}
		\\
		$T_6$
	\end{subfigure}
	\begin{subfigure}[b]{0.15\textwidth}
		\centering
		\includegraphics[width=0.56\linewidth]{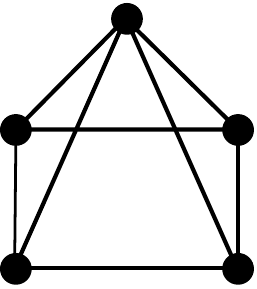}
		\\
		$T_7$
	\end{subfigure}
	\begin{subfigure}[b]{0.15\textwidth}
		\centering
		\includegraphics[width=0.56\linewidth]{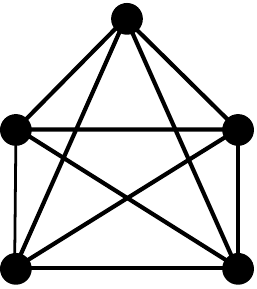}
		\\
		$T_8$
	\end{subfigure}
	\begin{subfigure}[b]{0.15\textwidth}
		\centering
		\includegraphics[width=0.7\linewidth]{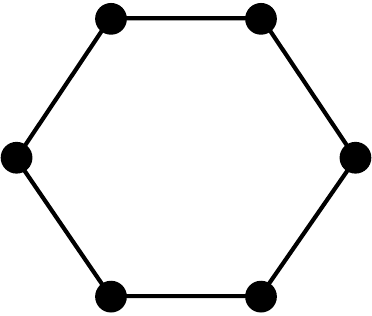}
		\\
		$T_9$
	\end{subfigure}
	\begin{subfigure}[b]{0.15\textwidth}
		\centering
		\includegraphics[width=0.56\linewidth]{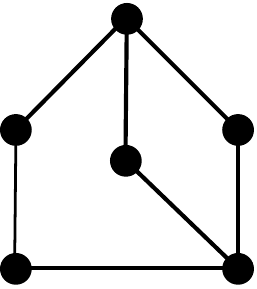}
		\\
		$T_{10}$
	\end{subfigure}
	
	\vspace{5mm}
	
	\begin{subfigure}[b]{0.15\textwidth}
		\centering
		\includegraphics[width=0.7\linewidth]{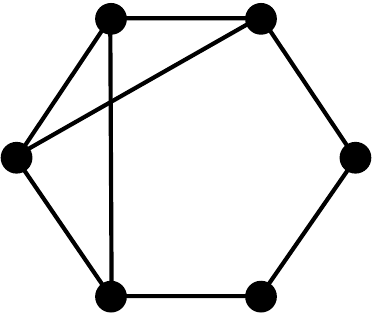}
		\\
		$T_{11}$
	\end{subfigure}
	\begin{subfigure}[b]{0.15\textwidth}
		\centering
		\includegraphics[width=0.7\linewidth]{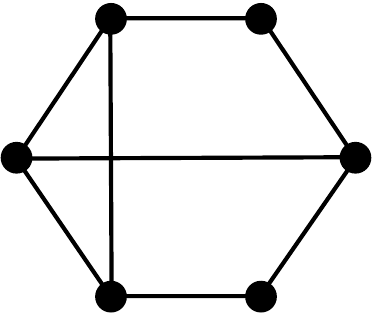}
		\\
		$T_{12}$
	\end{subfigure}
	\begin{subfigure}[b]{0.15\textwidth}
		\centering
		\includegraphics[width=0.7\linewidth]{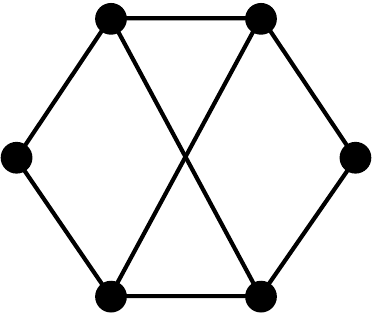}
		\\
		$T_{13}$
	\end{subfigure}
	\begin{subfigure}[b]{0.15\textwidth}
		\centering
		\includegraphics[width=0.8\linewidth]{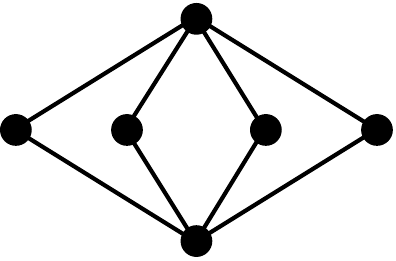}
		\\
		$T_{14}$
	\end{subfigure}
	\begin{subfigure}[b]{0.15\textwidth}
		\centering
		\includegraphics[width=0.7\linewidth]{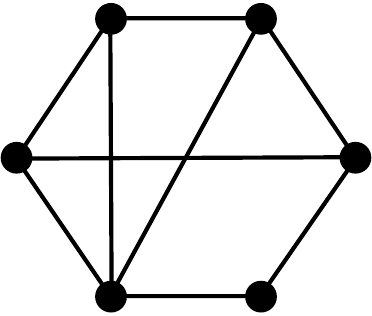}
		\\
		$T_{15}$
	\end{subfigure}
	
	\vspace{5mm}

	\begin{subfigure}[b]{0.15\textwidth}
		\centering
		\includegraphics[width=0.7\linewidth]{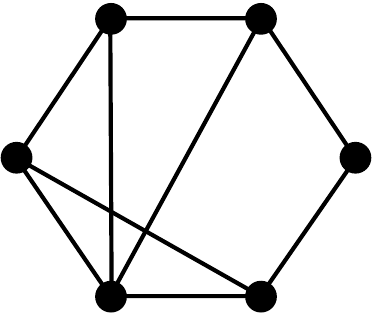}
		\\
		$T_{16}$
	\end{subfigure}    
	\begin{subfigure}[b]{0.15\textwidth}
		\centering
		\includegraphics[width=0.7\linewidth]{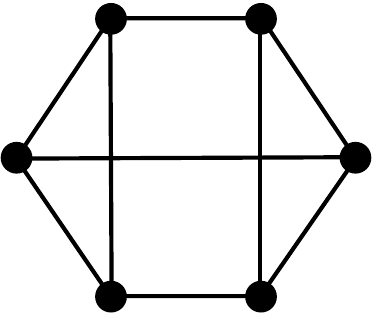}
		\\
		$T_{17}$
	\end{subfigure}
	\begin{subfigure}[b]{0.15\textwidth}
		\centering
		\includegraphics[width=0.7\linewidth]{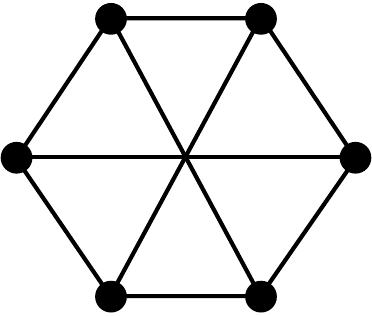}
		\\
		$T_{18}$
	\end{subfigure}
	\begin{subfigure}[b]{0.15\textwidth}
		\centering
		\includegraphics[width=0.7\linewidth]{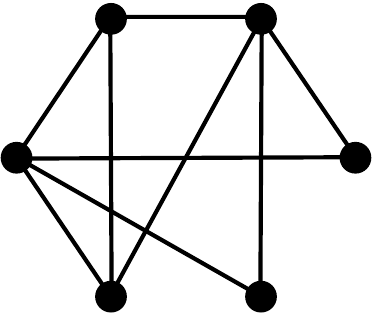}
		\\
		$T_{19}$
	\end{subfigure}
	\begin{subfigure}[b]{0.15\textwidth}
		\centering
		\includegraphics[width=0.7\linewidth]{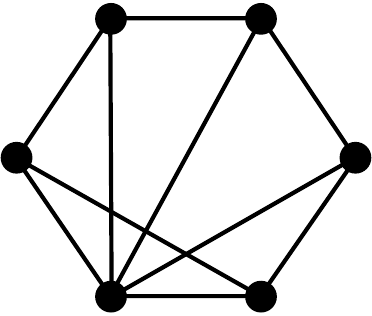}
		\\
		$T_{20}$
	\end{subfigure}
	\caption{Graphs $T_{1}$ through $T_{20}$ involved in the computation of the 5th and 6th coefficients of the chromatic polynomial \cite{Bielak1}.}\label{tgraphs1}
\end{figure} 

\begin{thm}\label{FarBie}\cite{Farrell, Bielak1}
	Let $G$ be a graph with $v$ vertices, $E$ edges, $t_1$ triangles, $t_2$ induced 4-cycles, and $t_3$ complete graphs of order 4. The first four coefficients of the chromatic polynomial $P_G(\la)$ are given by the following formulas: $c_v = 1$, $c_{v-1} = -E$, $c_{v-2} = \displaystyle\binom{E}{2}-t_1$, and 
	
	\begin{equation*}\label{chC3}
	    c_{v-3} = -\displaystyle\binom{E}{3}+(E-2)t_1+t_2-2t_3
	\end{equation*}

	The 5th and 6th coefficients are given by the following formulas, where $t_i$ is the number of induced subgraphs of $G$ isomorphic to graphs $T_i$ as shown in Figures \ref{tgraphs1} and  \ref{tgraphs2}.
	\begin{equation*}\label{chC4}
	c_{v-4} =\binom{E}{4} - \binom{E-2}{2}t_1 + \binom{t_1}{2} - (E-3)t_2 -(2E-9)t_3
- t_4 +t_5 + 2t_6+3t_7-6t_8	\end{equation*}
		\begin{eqnarray*}\label{chC5}
	c_{v-5} &=& -\binom{E}{5} + \binom{E-2}{3}t_1 - (E-4)\binom{t_1}{2} + \binom{E-3}{2}t_2 - (t_2-2t_3)t_1 - (E^2-10E+30)t_3
	\\&&+t_4 -(E-3)t_5-2(E-5)t_6-3(q-6)t_7+6(E-8)t_8+t_9-t_{10}
	-2t_{11}-2t_{12}-t_{13}\\&&+t_{14}
	-t_{15}-3t_{16}-4t_{17}-4t_{18}+2t_{19}-4t_{20}-t_{21}+4t_{22}
	+3t_{23}+4t_{24}+5t_{25}+4t_{26}\\&&+6t_{27}+8t_{28}
	+16t_{29}+12t_{30}-24t_{31}
	\end{eqnarray*}
\end{thm}

\begin{figure}[!ht]
	\centering
	\begin{subfigure}[b]{0.15\textwidth}
		\centering
		\includegraphics[width=0.7\linewidth]{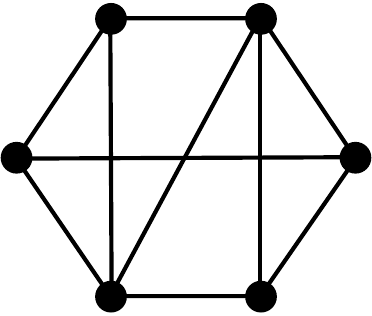}
		\\
		$T_{21}$
	\end{subfigure}
	\begin{subfigure}[b]{0.15\textwidth}
		\centering
		\includegraphics[width=0.8\linewidth]{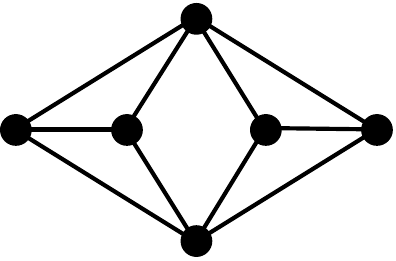}
		\\
		$T_{22}$
	\end{subfigure}
	\begin{subfigure}[b]{0.15\textwidth}
		\centering
		\includegraphics[width=0.8\linewidth]{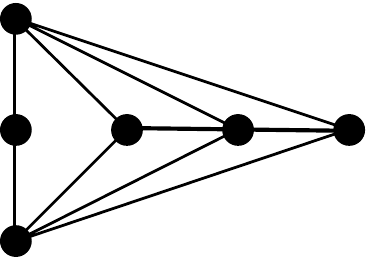}
		\\
		$T_{23}$
	\end{subfigure}
	\begin{subfigure}[b]{0.15\textwidth}
		\centering
		\includegraphics[width=0.7\linewidth]{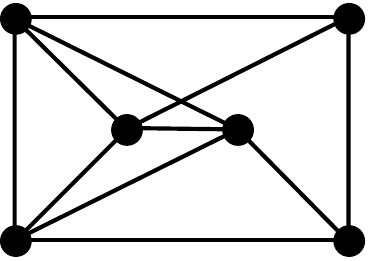}
		\\
		$T_{24}$
	\end{subfigure}    
	\begin{subfigure}[b]{0.15\textwidth}
		\centering
		\includegraphics[width=0.8\linewidth]{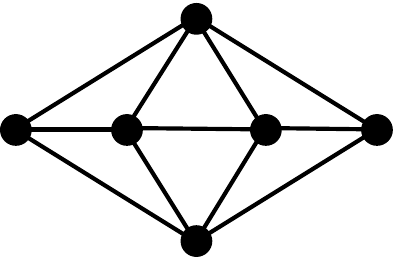}
		\\
		$T_{25}$
	\end{subfigure}
	
	\vspace{5mm}
	
	\begin{subfigure}[b]{0.15\textwidth}
		\centering
		\includegraphics[width=0.7\linewidth]{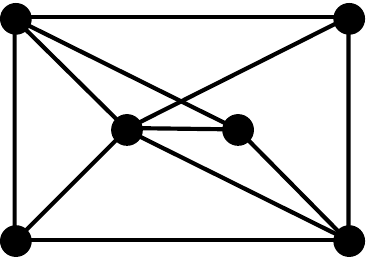}
		\\
		$T_{26}$
	\end{subfigure}
	\begin{subfigure}[b]{0.15\textwidth}
		\centering
		\includegraphics[width=0.7\linewidth]{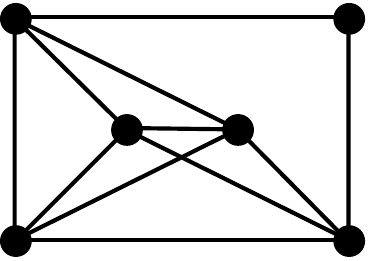}
		\\
		$T_{27}$
	\end{subfigure}
	\begin{subfigure}[b]{0.15\textwidth}
		\centering
		\includegraphics[width=0.7\linewidth]{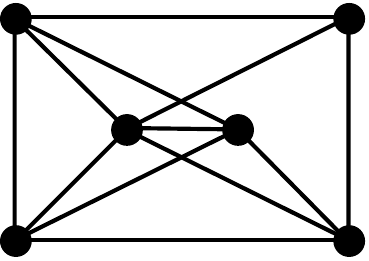}
		\\
		$T_{28}$
	\end{subfigure}
	\begin{subfigure}[b]{0.15\textwidth}
		\centering
		\includegraphics[width=0.7\linewidth]{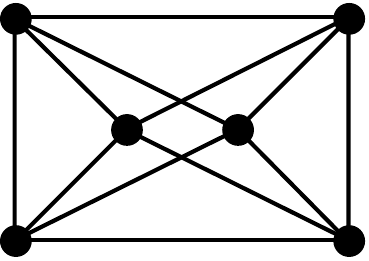}
		\\
		$T_{29}$
	\end{subfigure}
	\begin{subfigure}[b]{0.15\textwidth}
		\centering
		\includegraphics[width=0.7\linewidth]{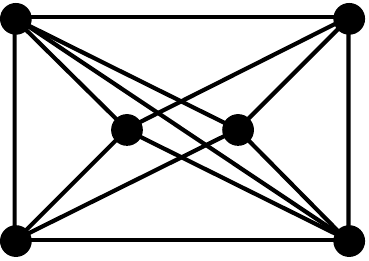}
		\\
		$T_{30}$
	\end{subfigure}
	\begin{subfigure}[b]{0.15\textwidth}
		\centering
		\includegraphics[width=0.7\linewidth]{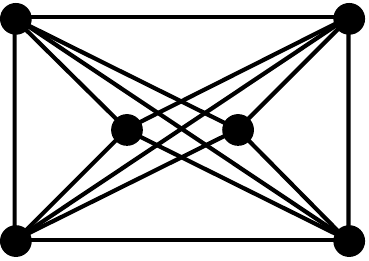}
		\\
		$T_{31}$
	\end{subfigure}
	\caption{Graphs $T_{21}$ through $T_{31}$ involved in the computation of the 5th and 6th coefficients of the chromatic polynomial \cite{Bielak1}.}\label{tgraphs2}
	
\end{figure}


\subsection{Khovanov and chromatic homology and their relations}\label{KhChCoeff}


The Jones polynomial has been categorified as the Euler characteristic of a bigraded homology theory known as Khovanov homology. We denote the Khovanov homology of a link by $Kh(L)$. The chromatic polynomial has a similar categorification known as chromatic graph homology. An overview of these homologies and their construction can be found in \cite{LS}, \cite{SaSco}. In this paper, we will use only the version of chromatic homology defined over $\A_2 = \Z[x]/(x^2)$ and will refer to it as $H_{\A_2}(G)$. Since $H_{\A_2}$ contains only  $\Z_2$ torsion \cite{LS}, we introduce the following notation.

\begin{defn}\label{2torsion}
If $H$ is a subgroup of either Khovanov or chromatic homology, $\tor H$ denotes the order 2 torsion subgroup of $H$. We use $\rktor H$ to indicate the number of copies of $\Z_2$.
\end{defn}

There is a partial correspondence between Khovanov homology of a link and the chromatic homology $H_{\A_2}$ of an associated graph.

\begin{thm}\cite{Przy1, PS} \label{Correspondence}
	Let $D$ be an oriented diagram of link $L$ with $c_-$ negative crossings and $c_+$ positive crossings. Suppose $G_+(D)$ has $v$ vertices and positive girth $\ell$. Let $p = i - c_-$ and $q = v - 2j + c_+ - 2c_-$. For $0 \le i < \ell$ and $j \in \Z$, there is an isomorphism $$H_{\A_2}^{i,j}(G_+(D)) \cong Kh^{p,q}(L).$$ Additionally, for all $j \in \Z$, there is an isomorphism of torsion: $\tor H_{\A_2}^{\ell,j}(G_+(D)) \cong \tor Kh^{\ell-c_-,q}(L).$
\end{thm}

Chromatic homology $H_{\A_2}(G)$ is always homologically thin (all non-trivial homology lies on two diagonals). If $Kh(L)$ is homologically thin, then $Kh(L)$ also contains only $\Z_2$ torsion \cite{Shum2}.

\begin{thm}\cite{LS} \label{ChromaticDetermined}
The chromatic homology $H_{\A_2}(G)$ with coefficients in $\Z$ is entirely determined by the chromatic polynomial $P_G(\la)$.
\end{thm}

Note that $P_G(\la)$ and $H_{\A_2}(G)$ are both trivial if $G$ contains a loop, and both ignore the presence of multiple edges in $G$. If $G$ is a loopless graph ($\ell(G)>1$) then both $G$ and its simplification $G'$ have the same chromatic invariants: $P_G(\la) = P_{G'}(\la)$ and $H_{\A_2}(G) = H_{\A_2}(G')$. If $\ell(G)>2$, then we also have $G=G'$. 

Theorem \ref{Correspondence} allows us to compute explicit formulae for extremal gradings of Khovanov homology, subject to combinatorial conditions on the Kauffman state of a link diagram. In \cite{AP,PPS,PS}, the following gradings of Khovanov homology are explicitly computed for diagrams when the isomorphism theorem holds.

\begin{prop}[\cite{PPS,PS}] \label{PS}
	Let $D$ be a diagram of $L$ with $c_+$ positive crossings, $c_-$ negative crossings, and $|s_+|$ circles in the all-positive Kauffman state of $D$. Let $N = -|s_+|+c_{+}-2c_{-}$ and let $p_1, t_1$ denote the cyclomatic number and number of triangles in $G_+(D)'$, respectively.
	
	If the girth of $G_+(D)$ is at least 2, then extreme Khovanov homology groups are given by:
 \begin{eqnarray*}
    	&&Kh^{-c_{-},N}(L) = \Z\hspace{1cm}
		Kh^{-c_{-},N+2}(L) = \begin{cases}
			\Z & G_+(D) \text{ bipartite}\\
			0 & \text{ otherwise}
		\end{cases} \\
		&& Kh^{1-c_{-},N+2}(L) =\begin{cases}
			\Z^{p_1} & G_+(D) \text{ bipartite}\\
			\Z^{p_1-1} \oplus \Z_2 & \text{ otherwise}
		\end{cases}
		 \end{eqnarray*}
If, in addition, the girth of $G_+(D)$ is at least 3, then we have an additional grading in Khovanov homology:
		 \begin{eqnarray*}
		Kh^{2-c_{-},N+4}(L)  &=&  \begin{cases}
			\Z^{\binom{p_1}{2}} \oplus \Z_2^{p_1} & G_+(D) \text{ bipartite}\\
			\Z^{\binom{p_1}{2}-t_1+1} \oplus \Z_2^{p_1-1}& \text{ otherwise}
		\end{cases}
 \end{eqnarray*}
\end{prop}

The following result is a restatement of \cite[Thm. 5.4]{SaSco}, describing the fourth and fifth homological gradings of $Kh(L)$ in terms of the associated graph.

\begin{thm}\cite{SaSco}\label{rankKhold}
Let $D$ be a diagram of $L$ as in Proposition \ref{PS}. Using conventions from Theorem \ref{FarBie} under the assumption that $G_+(D)$ has girth at least 4 we have the following gradings in Khovanov homology: 
\begin{eqnarray*}
\rk Kh^{3-c_-,N+6}(L) = \rk \tor Kh^{4-c_-,N+8}(L) = \begin{cases}
p_1 + \binom{p_1+1}{3}-t_2 & G_+(D) \text{ bipartite,}\\
p_1 + \binom{p_1+1}{3}-t_2-1 & \text{ otherwise.}
\end{cases}
\end{eqnarray*}
\end{thm}

\section{Extremal Khovanov homology computations}

In this section we rely on ideas used in Theorem \ref{rankKhold} to obtain explicit formulas for Khovanov homology in several additional extremal gradings using the formulas found in Theorem \ref{FarBie}. This approach can theoretically be extended to further groups on the diagonal using the method of \cite{Bielak1} to find formulas for additional chromatic coefficients, although this appears computationally challenging.

\begin{thm}\label{rankKhnew}
Let $D$ be a diagram of $L$ as in Proposition \ref{PS}. Suppose also that $G_+(D)$ has girth at least 5 with cyclomatic number $p_1$ and subgraphs $T_i$ denoted as in Theorem \ref{FarBie}. Let the coefficients $a_{v-4}$ and $a_{v-5}$ be as in Theorem \ref{456HGR}. Then we have the following relations in the Khovanov homology of $L$:
\begin{equation*}
 \rk Kh^{4-c_-,N+8}(L)=
 \rk \tor Kh^{5-c_-,N+10}(L) =   \begin{cases}
 \binom{p_1}{2}+ a_{v-4} & G_+(D) \text{ bipartite} \\
 \binom{p_1}{2}+ a_{v-4} + 1 & \text{ otherwise}
  \end{cases}
 \end{equation*}

If in addition, the girth of $G_+(D)$ is at least 6, then we also have the following:
 
 \begin{eqnarray*}
\rk Kh^{5-c_-,N+10}(L) &=&  \rk \tor Kh^{6-c_-,N+12}(L) =  \begin{cases}
p_1 + \binom{p_1+1}{3} - a_{v-5} & G_+(D) \text{ bipartite} \\
p_1 + \binom{p_1+1}{3}- a_{v-5} - 1 & \text{ otherwise}
 \end{cases}
\end{eqnarray*}
\end{thm}

Theorem \ref{rankKhnew} is an immediate consequence of Theorem \ref{456HGR} and the isomorphism theorem for diagrams whose all-positive graphs have girth at least 5 or 6.

\begin{thm}\label{456HGR}
Let $G$ be a simple graph with cyclomatic number $p_1$ and subgraphs $T_i$ denoted as in Theorem \ref{FarBie}. Then we have the following groups in the chromatic homology of $G$:
\begin{eqnarray*}\label{RkTor45}
\rk H_{\A_2}^{4,v-4}(G)  &=& \rk \tor H_{\A_2}^{5,v-5}(G) =  \begin{cases}
\binom{p_1}{2}+ a_{v-4} & G \text{ bipartite} \\
\binom{p_1}{2}-t_1+1+ a_{v-4} & \text{ otherwise}
 \end{cases}\\
\rk H_{\A_2}^{5,v-5}(G) &=& \rk \tor H_{\A_2}^{6,v-6}(G) =  \begin{cases}
p_1 + \binom{p_1+1}{3}-t_2 - a_{v-5} & G \text{ bipartite} \\
p_1 + \binom{p_1+1}{3}-t_1(p_1-1)-t_2+2t_3-1 - a_{v-5} & \text{ otherwise}
 \end{cases}
\end{eqnarray*}
     The coefficients $a_{v-4}$ and $a_{v-5}$ are given by:
      
\begin{align*}
a_{v-4} &= \binom{v}{4}-E\binom{v-1}{3}+ \left(\binom{E}{2}-t_1\right)\binom{v-2}{2}+c_{v-3}(v-3) + c_{v-4}\\
a_{v-5} &= \binom{v}{5}-E\binom{v-1}{4}+ \left(\binom{E}{2}-t_1\right)\binom{v-2}{3}+c_{v-3}\binom{v-3}{2}+ c_{v-4}(v-4)+c_{v-5}
\end{align*}
 \end{thm}

\begin{proof}

Let the chromatic polynomial of $G$ have the form given in Equation \ref{ChromCoef}. We change variables to $\la = q+1$ to match the graded Euler characteristic of $H_{\A_2}(G)$. The coefficient of $q^{v-i}$ in this polynomial will be denoted $a_i$.
\begin{align*}
P_G(q) &= (q+1)^v + c_{v-1}(q+1)^{v-1} + \ldots + c_2 (q+1)^2 + c_1 (q+1)\\
&= q^v + a_{v-1}q^{v-1} + \ldots + a_2 q^2 + a_1q + a_0.
\end{align*}
We proceed as in the proof of \cite[Thm. 5.3]{SaSco}, using the formulas for the $c_i$s in Theorem \ref{FarBie} and the equivalence of $P_G(q)$ with chromatic homology. Note that $t_1 = t_3 = 0$ if $G$ is bipartite.
\end{proof}

  \begin{table}[H] 
	\centering
	\renewcommand{\arraystretch}{1}
	\begin{tabular}{|c| P{0.6cm} | P{0.6cm} | P{0.6cm} | P{0.7cm} | P{0.6cm} | P{0.6cm} | P{1.7cm} }
		\cline{1-7}
		$\mathbf{j/i}$ & 0 & 1 & $\ldots$ & $\ell-1$ & $\ell$ & $\cdots$  \\ \cline{1-7}
		 $v$ & $\blacksquare$ &  &  &  &  & & $\rightarrow a_v$\\ \cline{1-7}
		$v-1$ & $\blacksquare$  & $\blacksquare$ &  &  &  & & $\rightarrow a_{v-1}$\\ \cline{1-7}
		 $\vdots$ &  & $\ddots$ & $\ddots$ &  &  & & $\vdots$  \\ \cline{1-7}
		 $v-\ell+1$ &  &  & $\blacksquare$ & $\blacksquare$ &   & & $\rightarrow a_{v-\ell+1}$ \\ \cline{1-7}
		 $v-\ell$ &  &  &  & $\blacksquare$ & $\square$ &  \\ \cline{1-7}
		 $\vdots$ &  &  &  &  & $\square$  & $\ddots$ \\ \cline{1-7}
	\end{tabular}
	\caption{Chromatic homology $H_{\A_2}(G_+(D))$ with coefficients $a_{v-i}$ for the chromatic polynomial on the right. $\square$ indicates possible homology. $\blacksquare$ indicates isomorphism with Khovanov homology.} \label{tablechromproof}
	
\end{table}

The following theorem completely describes the part of Khovanov homology which is obtained from the isomorphism in Theorem \ref{Correspondence}.

 \begin{thm}\label{rankKhovanovgirth}
	Let $D$ be a diagram of a link $L$ with $c_+$ positive crossings, $c_-$ negative crossings, and $|s_+|$ circles in the all-positive Kauffman state, and let $N = -|s_+|+c_{+}-2c_{-}$. Suppose that $G_+(D)$ satisfies the conditions of Theorem \ref{rankHGRgirth} (in particular, the  girth $\ell$ of $G_+(D)$ is greater than $2$). For $0 < i < \ell$, we have the following ranks of the Khovanov homology of $L$:
	\begin{equation} \label{KhFormula}
	\rk Kh^{i-c_-,N+2i}(L) = \left(\displaystyle\sum_{\substack{r \ge 0,\\ 0\le k = i-2r \le i}} \binom{p_1-2+k}{k} \right)-n_{i+1} +(-1)^{i+1} \delta^b
	\end{equation}
	
where $\delta^b = 1$ if $G_+(D)$ is bipartite and 0 otherwise.
\end{thm}

Based on Theorem \ref{Correspondence} and \cite{LS},  formula \eqref{KhFormula} also gives the number of $\Z_2$-torsion groups on the next grading of this diagonal: $\rk \tor Kh^{(i+1)-c_-,-|s_+|+c_+-2c_-+2(i+1)}(L)$. If we consider the all-negative state graph $G_-(D)$, an analogous statement holds for the highest homological gradings in $Kh(L)$.

\begin{cor} \label{RankGenFun}
    Let $D$ be a reduced diagram of $L$ that satisfies the conditions of Theorem \ref{rankKhovanovgirth}.
    If in addition, $G_+(D)$ is a non-bipartite graph, then the sequence of ranks 
    \begin{equation}
    \{S_0, \ldots, S_{\ell-2}\} = \{\rk H_{\A_2}^{i,v-i}(G_+(D))\}_{0\le i \le \ell-2} = \{\rk Kh^{i-c_-,N+2i}(L)\}_{0 \le i \le \ell-2}    
    \end{equation} is given by the first $\ell-1$ coefficients of the generating function $\dfrac{1}{(1+x)(1-x)^{p_1}}$.
\end{cor}

For graphs of girth $\ell$, Theorem \ref{Meredith} provides a succinct description of the first $\ell$ coefficients of the chromatic polynomial. We first translate this statement into a description of the ranks of chromatic homology in the first $\ell$ homological gradings. As a corollary, we obtain the entire part of Khovanov homology that is determined by the all-positive or all-negative state graph as in Theorem \ref{rankKhovanovgirth}.

\begin{thm}\label{rankHGRgirth}
	Let $G$ be a simple graph with girth $\ell>2$, cyclomatic number $p_1$, and   $n_{i}$ denoting the number of $i$-cycles in $G$. Then, for $0 < i < \ell$, we have the following ranks of the chromatic homology of $G$:
	\begin{equation} \label{eqnrkHGR}
		\rk H_{\A_2}^{i,v-i}(G) = \left(\displaystyle\sum_{\substack{r \ge 0,\\ 0\le k = i-2r \le i}} \binom{p_1-2+k}{k} \right)-n_{i+1} +(-1)^{i+1} \delta^{b}
	\end{equation}
	where $\delta^{b} = 1$ if $G$ is bipartite and 0 otherwise.
\end{thm}

\begin{proof}
	For $i=1,2,3$, this statement follows from \cite{PS}, \cite{SaSco}. We show by induction that it holds for $i>3$.
	
	 As above, let $a_{v-i}$ denote the coefficient of $P_G(q)$ derived from the quantum grading $j=v-i$ (see Table \ref{tablechromproof}). For $0 < i < \ell-1$, we use Theorem \ref{Meredith} to compute:
	\begin{align*}
	a_{v-i} &= \binom{v}{v-i} + c_{v-1}\binom{v-1}{v-i} + \ldots + c_{v-i}\binom{v-i}{v-i} = \displaystyle\sum_{k=0}^i c_{v-k} \binom{v-k}{v-i}\\
	&= \displaystyle\sum_{k=0}^i (-1)^{k} \displaystyle\binom{E}{k} \binom{v-k}{v-i} = (-1)^i \binom{p_1-2+i}{i}
	\end{align*}
	and for $i= \ell-1$, a similar computation shows that $a_{v-(\ell-1)} = (-1)^{\ell-1} \left( \binom{p_1-2+(\ell-1)}{\ell-1} - n_{\ell} \right).$ For $i < \ell-1$, we have $n_{i+1} = 0$, and thus we can say $$a_{v-i} = (-1)^i\left( \binom{p_1-2+i}{i} - n_{i+1} \right)$$ for $0<i \le \ell-1$.
	
	Suppose that $3 < i < \ell$ and that Equation \ref{eqnrkHGR} holds for all homological gradings less than $i$. We show that Equation \ref{eqnrkHGR} also holds for $\rk H_{\A_2}^{i,v-i}(G)$. Since chromatic homology is thin, each coefficient is the difference of the ranks of the two homology groups in the grading $j=v-i$.
	
	\begin{equation} \label{RankDiff}
		a_{v-i} = (-1)^{i-1}\rk H_{\A_2}^{i-1,v-i}(G) + (-1)^{i}\rk H_{\A_2}^{i,v-i}(G)
	\end{equation}

	 By the knight move isomorphism of \cite{CCR}, $\rk H_{\A_2}^{i-1,v-i}(G) = \rk H_{\A_2}^{i-2,v-(i-2)}(G)$. We make this substitution into Equation \ref{RankDiff}, along with the value of $a_{v-i}$ derived above.
	 
	\begin{align*}
	     (-1)^i\left( \binom{p_1-2+i}{i} - n_{i+1} \right) &= (-1)^{i-1}\rk H_{\A_2}^{i-2,v-(i-2)}(G) + (-1)^{i}\rk H_{\A_2}^{i,v-i}(G)\\
	    \binom{p_1-2+i}{i} - n_{i+1}  &= -\rk H_{\A_2}^{i-2,v-(i-2)}(G) + \rk H_{\A_2}^{i,v-i}(G)\\
	     \rk H_{\A_2}^{i,v-i}(G) &= \rk H_{\A_2}^{i-2,v-(i-2)}(G) + \binom{p_1-2+i}{i}- n_{i+1}
	\end{align*}
		By the induction assumption
	\begin{align*}
		\rk H_{\A_2}^{i-2,v-(i-2)}(G) = \left(\displaystyle\sum_{\substack{r  \ge 0,\\ 0 \le k = (i-2)-2r \le i-2}} \binom{p_1-2+k}{k} \right)-n_{i-1} +(-1)^{i-1} \delta^b
	\end{align*}
	We may drop the term $n_{i-1} = 0$ since we are assuming $i < \ell$. Finally, we collect all binomial coefficients into the summation and note that $i-1$ has the same parity as $i+1$.
	\begin{align*}
	 \rk H_{\A_2}^{i,v-i}(G) &= \left(\displaystyle\sum_{\substack{r  \ge 0,\\ 0 \le k = (i-2)-2r \le i-2}} \binom{p_1-2+k}{k} \right) +(-1)^{i-1} \delta^b + \binom{p_1-2+i}{i} - n_{i+1}\\
	&=  \left(\displaystyle\sum_{\substack{r  \ge 0,\\ 0 \le k = i-r \le i}} \binom{p_1-2+k}{k} \right)- n_{i+1} + (-1)^{i+1} \delta^b \qedhere
	\end{align*}
\end{proof}

\begin{exa}\label{KhFormulaEx}
Let $K$ be the knot $11a362$ (Dowker-Thistlethwaite notation) with diagram $D$ depicted in Figure \ref{diagram11}. The all-positive state graph $G_+(D)$ has girth $\ell = 6$ and cyclomatic number $p_1 = 2$. The Khovanov homology $Kh(K)$ is shown in Table \ref{exampletable1} and the chromatic homology $H_{\A_2}(G_+(D))$ in Table \ref{exampletable2}. 

 \begin{figure}[h]
	\centering
	\includegraphics[width=0.7\textwidth]{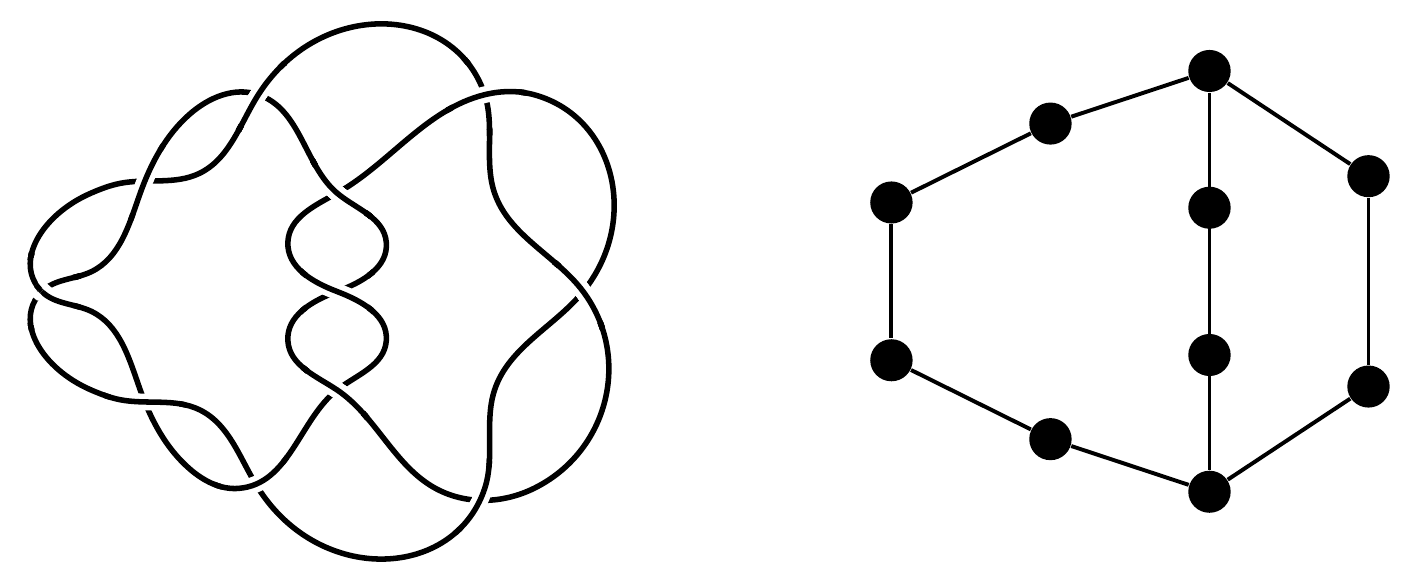}
	\caption{Diagram of $11a362$ with all-positive state graph $G_+(D)$}
	    \label{diagram11}
\end{figure}

The graph $G_+(D)$ is bipartite with $c_- = 11$ and $N = -32$. The groups shown in bold in Table \ref{exampletable1} are those which correspond to chromatic homology groups in $H_{\A_2}(G_+(D))$. Theorem \ref{rankKhovanovgirth} describes the ranks of these Khovanov homology groups which are located on the lower diagonal. For $i=1$ through $i=\ell-2=4$:
	\begin{align*}
	    	\rk Kh^{i-c_-,N+2i}(L) &= \left(\displaystyle\sum_{\substack{r \ge 0,\\ 0\le k = i-2r \le i}} \binom{2-2+k}{k} \right) +(-1)^{i+1} \delta^b \\ &= \left(\displaystyle\sum_{\substack{r \ge 0,\\ 0\le k = i-2r \le i}} 1 \right) +(-1)^{i+1} = \left \lfloor \dfrac{i}{2}+1 \right \rfloor +(-1)^{i+1}
	\end{align*}
while for $i=\ell-1 = 5$:
	\begin{align*}
	    	\rk Kh^{5-c_-,N+10}(L) = \left(\displaystyle\sum_{\substack{r \ge 0,\\ 0\le k = i-2r \le i}} 1 \right) -n_6 +(-1)^{5+1} = \left \lfloor \dfrac{5}{2}+1 \right \rfloor -1 +(-1)^{6} = 3
	\end{align*}
	
Observe that if one ignores the $\delta^b$ term that keeps track of the bipartite property, the first $\ell-1$ ranks given by the formula are $1, 1, 2, 2, 3$ which are the first 5 coefficients of the generating function $\dfrac{1}{(1+x)(1-x)^{2}}$ (see Corollary \ref{RankGenFun}).

	\begin{center}
\begin{table}
	\renewcommand{\arraystretch}{1}
	\begin{center}
	\begin{tabular}{|c| P{0.6cm} | P{0.6cm} | P{0.8cm} | P{0.8cm} | P{0.8cm} | P{0.8cm} | P{0.9cm} | P{0.8cm} | P{0.8cm} | P{0.8cm} | P{0.5cm} | P{0.5cm} | P{0.5cm} }
		\cline{1-13}
		$\mathbf{q/p}$ & $-11$ & $-10$ & $-9$ & $-8$ & $-7$ & $-6$ & $-5$ & $-4$ & $-3$ & $-2$ & $-1$ & $0$ &  \\ \cline{1-13}
		 $-8$ & &  &  &  &  & & & & & & & $1$ &  \\ \cline{1-13}
		  $-10$ & &  &  &  &  & & & & & & & $1_2$ &  \\ \cline{1-13}
		   $-12$ & &  &  &  &  & & & & &  $3$ & $1$ & &  \\ \cline{1-13}
		 $-14$ & &  &  &  &  & & & & $1$ & $3_2$ & & &  \\ \cline{1-13}
		$-16$  &   &  &  & &  &  &  & $3$ & $3$, $1_2$ && & &  \\ \cline{1-13}
		$-18$ &  &  &   &  &  &  & $3$ & $1$, $3_2$ & && & &  \\ \cline{1-13}
		$-20$ &  &  &    & &   & $\fakebold 2$ & $3$, $\fakebold 3_{\textbf{2}} $ & & & & & &  \\ \cline{1-13}
		$-22$ &   &  &   &  & $\fakebold 3$ & $\fakebold 3, \fakebold 2_{\textbf{2}} $ & & && & & &  \\ \cline{1-13}
		$-24$ &  &  &  & $\fakebold 1$ & $\fakebold 2, \fakebold 3_{\textbf{2}} $  &  & & && & & &  \\ \cline{1-13}
		$-26$  &  &  & $\fakebold 2$ & $\fakebold 3, \fakebold 1_{\textbf{2}}$ &  &  & & & && & &  \\ \cline{1-13}
		$-28$ &  &  & $\fakebold 1, \fakebold 2_{\textbf{2}}$ & &   &  & & & & & & &  \\ \cline{1-13}
		$-30$ &  $\fakebold 1$ & $\fakebold 2$ &   &  &  &  & & && & & &  \\ \cline{1-13}
		 $-32$ & $\fakebold 1$ &  &  &  &   &  & & && & & &  \\ \cline{1-13}
	\end{tabular}
	\caption{Khovanov homology $Kh(11a362)$. An entry of $k$ represents a summand $\Z^k$, and $k_{2}$ represents a summand of $\Z_2^k$. Entries in bold, from -11 to -5, are isomorphic to gradings in chromatic homology of $G_+(D)$ via Theorem \ref{Correspondence}.} \label{exampletable1}
	\end{center}
\end{table}\end{center}

	\begin{center}
\begin{table}
	\renewcommand{\arraystretch}{1}
	\begin{center}
	\begin{tabular}{|c| P{0.6cm} | P{0.6cm} | P{0.8cm} | P{0.8cm} | P{0.8cm} | P{0.8cm} | P{0.9cm} | P{0.8cm} | P{0.8cm} | P{0.8cm} }
		\cline{1-10}
		$\mathbf{j/i}$ & $0$ & $1$ & $2$ & $3$ & $4$ & $5$ & $6$ & $7$ & $8$ &  \\ \cline{1-10}
		 $10$ & $\fakebold 1$ &  &  &  &  & & & & &    \\ \cline{1-10}
		  $9$ & $\fakebold 1$ & $\fakebold 2$ &  &  &  & & & &  &  \\ \cline{1-10}
		   $8$ & & & $\fakebold 1, \fakebold 2_{\textbf{2}}$  &  &  & & &  &  &   \\ \cline{1-10}
		 $7$ & &  & $\fakebold 2$ & $\fakebold 3, \fakebold 1_2$ &  & &  & &  &   \\ \cline{1-10}
		$6$  &   &  &  & $\fakebold 1$ & $\fakebold 2, \fakebold 3_{\textbf{2}}$ &    &  &   &&  \\ \cline{1-10}
		$5$ &  &  &   &  &   $\fakebold 3$ & $\fakebold 3, \fakebold 2_{\textbf{2}}$ &  & & &   \\ \cline{1-10}
		$4$ &  &  &    &    &  & $\fakebold 2$ & $2, \fakebold 3_{\textbf{2}}$ & & &  \\ \cline{1-10}
		$3$ &   &     &  &  &  & & $3$ & $1, 2_2$ & &  \\ \cline{1-10}
		$2$ &    &  &  &   &  & & & $2$ & $1_2$ &  \\ \cline{1-10}
		$1$  &    &  &  &  &  & & & & $1$ & \\ \cline{1-10}
	\end{tabular}
	\caption{Chromatic homology $H_{\A_2}(G_+(D))$ for the graph in Example \ref{KhFormulaEx}. An entry of $k$ represents a summand $\Z^k$, and $k_{2}$ represents a summand of $\Z_2^k$. Entries in bold are isomorphic to gradings in Khovanov homology (see Table \ref{exampletable1}).} \label{exampletable2}
	\end{center}
\end{table}\end{center}
\end{exa}

\section{More on the head and tail of the Jones polynomial}

Theorem \ref{rankKhovanovgirth} can be used to compute the $\ell$ extremal coefficients in the head or tail of the Jones polynomial, subject to certain conditions on the Khovanov homology of $L$. As shown in Table \ref{table2}, if $Kh^{i,j}(L)$ is trivial for $i > (\ell-1) - c_-$ and $j < N+2\ell$, then the extremal coefficients of the unnormalized Jones polynomial are determined by precisely by the ranks described in Theorem \ref{rankKhovanovgirth}. We conjecture that these gradings are always trivial in Khovanov homology, so that Theorem \ref{Jones1} is true for all links and not only those which are Khovanov thin.

\begin{table}
	\centering
	\renewcommand{\arraystretch}{1}
	\begin{tabular}{|c|c | c | c | c | c|c |c }
		\cline{1-7}
		$\mathbf{q/p}$ & $-c_-$ & $1-c_-$ & $\ldots$ & $(\ell-1)-c_-$ & $\ell-c_-$ & $\cdots$  \\ \cline{1-7}
		$\vdots$ & &  &  &  & $\square$ & $\iddots$& \\ \cline{1-7}
		$N+2\ell$ &   &  &  & $\blacksquare$ &$\square$  &  & \\ \cline{1-7}
		$N+2(\ell-1)$ &  &  & $\blacksquare$  & $\blacksquare$ & ? & ? & $\rightarrow \al_{\ell-1} = a_{v-\ell+1}$ \\ \cline{1-7}
		$\vdots$ &  & $\iddots$ & $\iddots$   & &  ? & ? & $\vdots$ \\ \cline{1-7}
		$N+2$ & $\blacksquare$  & $\blacksquare$ &   &  & ? & ? & $\rightarrow \al_1 = a_{v-1}$ \\ \cline{1-7}
		$N$ & $\blacksquare$ &  &  &  & ?  & ? & $\rightarrow \al_0 = a_v$ \\ \cline{1-7}
	\end{tabular}
	\caption{Khovanov homology $Kh(L)$ with coefficients $\al_i$ for the unnormalized Jones polynomial on the right (compare with chromatic coefficients $a_{v-i}$ in Table \ref{tablechromproof}). $\square$ indicates possible homology. $\blacksquare$ indicates isomorphism with chromatic homology. ? indicates gradings that must be zero for results to hold for all links} \label{table2}	
\end{table}

\begin{thm} \label{Jones1}
Let $D$ be a diagram of a link $L$ such that $Kh(L)$ is homologically thin and $G_+(D)$ has girth $\ell > 2$, with cyclomatic number $p_1$ and number of $\ell$-cycles $n_{\ell}$. Then the first $\ell$ coefficients in the tail of $J_L$ are given by the formula: 
\begin{equation}\label{JC}
\be_{i} = \begin{cases}
(-1)^{i-c_-(D)}\binom{p_1-1+i}{i} & 0 \le i \le \ell - 2\\
(-1)^{\ell-1-c_-(D)}\left(\binom{p_1-1+(\ell-1)}{\ell-1} - n_{\ell}\right) & i = \ell-1\\
\end{cases}\end{equation}

If we consider the all-negative state graph $G_-(D)$, an analogous statement holds for the first $\ell(G_-(D))$ coefficients in the head of $J_L$.
\end{thm}

\begin{proof} The normalized Jones polynomial $J_L$ is the graded Euler characteristic of the reduced Khovanov homology $\Widetilde{Kh}(L)$, so we prove this result by describing $\Widetilde{Kh}(L)$. Let $R_i = \rk Kh^{i-c_-, N+2i}(L)$ given by the formula in Theorem \ref{rankKhovanovgirth}:
\begin{equation}
    R_i = \rk Kh^{i-c_-,N+2i}(L) = \left(\displaystyle\sum_{\substack{r \ge 0,\\ 0\le k = i-2r \le i}} \binom{p_1-2+k}{k} \right)-n_{i+1} +(-1)^{i+1} \delta^b
\end{equation}
for $0 < i < \ell$. We also assign $R_0 = 1$, the rank of $Kh^{-c_-, N}(L)$ for any $L$ satisfying the above condition on girth. For notational convenience we let $R_i = 0$ for $i<0$.

Since the extreme gradings of $Kh(L)$ from $i=c_-$ through $i=(\ell-1)-c_-$ are isomorphic to chromatic homology, this part of $Kh(L)$ is thin with only $\Z_2$ torsion. The knight move isomorphism gives us the ranks of the additional groups on the second diagonal:
\begin{equation}
    R_i = \rk Kh^{i-c_-,N+2i}(L) = \rk Kh^{(i+1)-c_-,N+2(i+2)}(L)
\end{equation}
which is valid for $1 \le i \le \ell-2$ if $G_+(D)$ is bipartite and $0 \le i \le \ell-2$ otherwise. Passing to Khovanov homology over $\Z_2$ coefficients, we can replace ``knight move" pairs of $\Z$s with ``tetrominos" of $\Z_2$s, see \cite{Shum2, LS}:
\begin{equation}
    \rk \tor Kh_{\Z_2}^{i-c_-,N+2i}(L) = \rk \tor Kh_{\Z_2}^{i-c_-,N+2(i+1)}(L) = \begin{cases}
    R_1 & i=1 \text{ and $G_+(D)$ bipartite} \\
    R_{i-1}+R_i & \text{ otherwise, } 0 \le i < \ell
    \end{cases}
\end{equation}

Using \cite[Cor.3.2C]{Shum1} and the fact that $\rk \Widetilde{Kh^{i,j}}(L) = \rk \Widetilde{Kh^{i,j}_{\Z_2}}(L)$ for thin links, we obtain the reduced integer Khovanov homology:
\begin{equation}
    \rk \Widetilde{Kh}^{i-c_-,N+2i+1}(L) = \begin{cases}
    R_1 & i=1 \text{ and $G_+(D)$ bipartite} \\
    R_{i-1}+R_i & \text{ otherwise, } 0 \le i < \ell
    \end{cases}
\end{equation}

Taking the graded Euler characteristic of $\Widetilde{Kh}(L)$ yields the following coefficients in the tail of $J_L(q)$:
\begin{equation}
\be_{i} = \begin{cases}
(-1)^{1-c_-}R_1 & i = 1 \text{ and $G_+(D)$ bipartite}\\
(-1)^{i-c_-}(R_{i-1}+R_i) & \text{ otherwise, } 0 \le i < \ell
\end{cases}\end{equation}

The formula in Equation \ref{JC} clearly holds for $i=0$, and for the bipartite case when $i=1$ by direct calculation.
\begin{align*}
     R_1 = \left(\displaystyle\sum_{\substack{r \ge 0,\\ 0\le k = 1-2r \le 1}} \binom{p_1-2+k}{k} +(-1)^2\right) = \binom{p_1-1}{1} + 1 = \binom{p_1-1+1}{1}
\end{align*}

For all other cases with $1 \le i < \ell-1$, we compute $R_{i-1} + R_i $:
\begin{align*}
& 
    \left(\displaystyle\sum_{\substack{r \ge 0,\\ 0\le k = (i-1)-2r \le i-1}} \binom{p_1-2+k}{k} +(-1)^{i}\delta_b\right)+ \left(\displaystyle\sum_{\substack{r \ge 0,\\ 0\le k = i-2r \le i}} \binom{p_1-2+k}{k} +(-1)^{i+1}\delta_b\right)\\
    &= \displaystyle\sum_{k=0}^i \binom{p_1-2+k}{k}= \binom{p_1-1+i}{i}=R_{i-1} + R_i
\end{align*}
where $\delta^{b} = 1$ if $G_+(D)$ is bipartite and 0 otherwise.
For $i=\ell-1$, $n_{i+1} = n_{\ell}$ is non-zero so that $R_{\ell-2} + R_{\ell-1} = \binom{p_1-1+(\ell-1)}{\ell-1} - n_{\ell}$.
\end{proof}
\begin{exa}\label{JonesFormulaEx}
Let $K$ be the knot $11a362$ as in Example \ref{KhFormulaEx}. 
The unnormalized Jones polynomial $\hat{J}_K(q)$ is the graded Euler characteristic of $Kh(K)$:
\begin{align*}
    \hat{J}_K(q) &= -q^{-32}+q^{-30}-q^{-28}+q^{-26}-q^{-24}-q^{-20}-2q^{-18}-q^{-14}+2q^{-12}+q^{-8}
\end{align*}
Table \ref{exampletablejones} illustrates how the first six terms in the tail (including one zero term) arise from the part of $Kh(K)$ that agrees with the chromatic homology of $G_+(D)$.
The normalized Jones polynomial is given by:
\begin{align*}
    &J_K(q) = \hat{J}_K(q)/(q+q^{-1})= \be_0 q^{-31} + \be_1 q^{-29} + \be_2 q^{-27} + \be_3 q^{-25} + \be_4 q^{-23} + \be_5 q^{-21} + \ldots\\
    &= -q^{-31} + 2q^{-29} - 3q^{-27} + 4q^{-25} - 5q^{-23} + 5q^{-21} - 6q^{-19} + 4q^{-17} - 4q^{-15} + 3q^{-13} - q^{-11} + q^{-9}
\end{align*}
Any A-adequate diagram of $K$ must have the same number of negative crossings $(c_- = 11)$ as diagram $D$ \cite{Lick}. The coefficients $\be_i$ agree with the formulas given in Theorem \ref{Jones1} for $0 \le i \le \ell-1 = 5$.

\begin{table}[H]
	\centering
	\renewcommand{\arraystretch}{1}
	\begin{tabular}{|c| P{0.8cm} | P{0.8cm} | P{0.8cm} | P{0.8cm} | P{0.8cm} | P{0.8cm} | P{0.9cm} | P{0.8cm} | l  }
		\cline{1-9}
		$\mathbf{q/p}$ & $-11$ & $-10$ & $-9$ & $-8$ & $-7$ & $-6$ & $-5$ & $-4$ &  \\  \cline{1-9}
		$-16$  &   &  &  & &  &  &  & $\iddots$ &   \\ \cline{1-9}
		$-18$ &  &  &   &  &  &  & $3$ & $\iddots$ &   \\ \cline{1-9}
		$-20$ &  &  &    & &   & $\fakebold 2$ & $3$, $\fakebold 3_2$ & &  \\ \cline{1-9}
		$-22$ &   &  &   &  & $\fakebold 3$ & $\fakebold 3, \fakebold 2_2$ & & & $\rightarrow 0$  \\ \cline{1-9}
		$-24$ &  &  &  & $\fakebold 1$ & $\fakebold 2, \fakebold 3_2$  &  & & & $\rightarrow -q^{-24}$ \\ \cline{1-9}
		$-26$  &  &  & $\fakebold 2$ & $\fakebold 3, \fakebold 1_2$ &  &  & & & $\rightarrow q^{-26}$ \\ \cline{1-9}
		$-28$ &  &  & $\fakebold 1, \fakebold 2_2$ & &   &  & & & $\rightarrow -q^{-28}$ \\ \cline{1-9}
		$-30$ &  $\fakebold 1$ & $\fakebold 2$ &   &  &  &  & & & $\rightarrow q^{-30}$ \\ \cline{1-9}
		 $-32$ & $\fakebold 1$ &  &  &  &   &  & & & $\rightarrow -q^{-32}$ \\ \cline{1-9}
	\end{tabular}
	\caption{Lowest gradings of the Khovanov homology $Kh(11a362)$. Entries in bold are isomorphic to gradings in chromatic homology of $G_+(D)$ via Theorem \ref{Correspondence}. The rightmost column contains the terms of the graded Euler characteristic which agree with the chromatic polynomial of $G_+(D)$.} \label{exampletablejones}
\end{table}

\end{exa}

\section{Girth of a link}\label{Girth1}

Each planar diagram $D$ of a link $L$ has an associated state graph $G_+(D)$ whose chromatic homology is related to Khovanov homology by the correspondence described in Theorem \ref{Correspondence}. 
Notice that the girth of $G_+(D)$ depends on the diagram of a knot. For example, adding a right-hand twist to any strand of $D$ using a Reidemeister I move creates a loop in $G_+(D)$, reducing the girth of this graph to $1$. 

The applicability of Theorem \ref{rankKhovanovgirth} depends on the girth of  $G_+(D)$; therefore, given any link $L$, we are interested in finding a diagram $D$ that maximizes the contribution of chromatic homology to Khovanov homology.

The largest such contribution made to $Kh(L)$ comes from the diagram for which $G_+(D)$ has the largest possible girth. This fact motivated the following definition that allows us to explicitly state the extent of the correspondence between Khovanov homology and the Jones polynomial with chromatic homology and the chromatic polynomial, respectively. 

\begin{defn}[\cite{SaSco}]
The girth of a link $L$ is $\gr(L) = \max\{\ell(G_+(D))~|~\text{$D$ is a diagram of L} \}$ where $G_+(D)$ is the graph obtained from the all-positive Kauffman state of diagram $D$, and $\ell(G_+(D))$ is the girth of graph $G_+(D)$.
\end{defn}

\begin{prop}[\cite{SaSco}] \label{GirthFinite}
	The girth $\gr(L)$ of any link $L$ is finite. 
\end{prop} 

\begin{proof}
    The non-trivial groups in $H_{\A_2}(G_+(D))$ span at least $\ell(G_+(D))-1$ homological gradings \cite{SaSco}, which are isomorphic to $\ell(G_+(D))-1$ corresponding gradings in $Kh(L)$. Since the span of $Kh(L)$ is bounded above, so are the possible girths of $G_+(D)$.
\end{proof}

Girth can alternatively be defined by taking the maximum value of $\ell(G_-(D))$ over all diagrams of $L$, or both values can be considered. A proof similar to that of Proposition \ref{GirthFinite} shows this invariant is also finite for any $L$.

In the rest of this section we analyze properties and bounds on girth coming from the Jones polynomial and Khovanov homology, as well as types of graphs that can appear as state graphs for diagrams of a given link.

\subsection{Bounds on the girth}
As with many knot invariants defined as a maximum or a minimum over all diagrams of a given knot, one bound is much easier to prove than the other.
In the case of girth, any knot diagram gives a lower bound. Theorems \ref{rankKhovanovgirth} and \ref{Jones1} provide some insight into what the upper bound on the girth of a link might be and the properties of a graph $G$ which realizes the girth.

\begin{cor} \label{KhUpper}
	Let $L$ be a link and let $M_K$ be the greatest number such that 
\begin{equation} \label{KhUpperEq}
    	\rk Kh^{P+i, Q+2i}(L) = 
    	\left(\displaystyle\sum_{\substack{r \ge 0,\\ 0\le k = i-2r \le i}} \binom{b-2+k}{k} \right) +(-1)^{i+1} \delta
\end{equation} for all $0 < i \leq M_K-2$, where $P$ and $Q$ are the lowest homological and quantum gradings in $Kh(L)$, $b>0$, and either $\delta = 0$ or $\delta = 1$ for all $i$. Then \begin{equation}\label{UpKhov}
	    \gr(L) \le M_K.
	\end{equation}
\end{cor}

\begin{cor} \label{JonesUpper}
	Suppose that link $L$ is Khovanov thin with Jones polynomial $J_L(q)$ as in Definition \ref{JonesCoef}. Let $M_J$ be the greatest number such that $|\be_i| = \binom{b-1+i}{i}$ for some b, with signs alternating, for all $0\leq i\leq M_J-2$. Then \begin{equation}\label{UpJon}
	    \gr(L) \le M_J.
	\end{equation}
\end{cor}

It is worth noting that with the above notation, $\gr(L) \le M_J \le M_K$. Moreover, $M_J = M_K$ for homologically thin links and we conjecture this will be true in general.

In Example \ref{InequalityEx}, we demonstrate that the upper bounds for $\gr(K)$ provided by Khovanov homology and the Jones coefficients are not necessarily achieved by any diagram of $K$.

\subsection{On all-positive state graphs}

The following corollary of Theorem \ref{rankKhovanovgirth} states that if Khovanov and chromatic homology agree on 3 or more gradings, this agreement imposes a restriction on the type of graphs that realize the isomorphism.

\begin{cor} \label{Khgraphproperty}
	Suppose that the link $L$ has 
	 a diagram $D$ such that  $G_+(D)$ has girth $\ell > 2$. Then:
	 \begin{enumerate}
	     \item \cite{PPS} $G_+(D)$ is bipartite if and only if $\rk Kh^{-c_-, N+2}(L) = 1$.
	     \item \cite{PPS} the cyclomatic number of $G_+(D)$ is $$p_1 = \rk Kh^{1-c_-, N+2}(L)-\rk Kh^{-c_-, N+2}(L)+1$$
	     \item 
	     the number of $\ell$-cycles in $G_+(D)$ is equal to 
	     \begin{align*}
	     n_{\ell} = \left(\sum_{\substack{r  \ge 0,\\ 0 \le k = (\ell-1)-2r \le (\ell-1)}} \binom{p_1-2+k}{k} \right) +(-1)^{\ell}\rk Kh^{-c_-,N+2}(L) - \rk Kh^{(\ell-1)-c_-, N+2(\ell-1)}(L)
	     \end{align*}
	 \end{enumerate} 
\end{cor}

A similar result exists for thin links via the Jones polynomial and Theorem \ref{Jones1}.

\begin{cor} \label{girthGraph}
	Suppose that link $L$ is homologically thin with Jones polynomial:
	$$J_L(q) = \be_0q^{C} + \be_1q^{C+2} + \be_2q^{C+4} + \be_3q^{C+6} + \ldots$$
	and that $D$ is a diagram of $L$ such that $G_+(D)$ has girth $\ell > 2$. Then the cyclomatic number of $G_+(D)$ is equal to $|\be_1|$ and the number of $\ell$-cycles in $G_+(D)$ is equal to $\binom{|\be_1|-1+(\ell-1)}{\ell-1} - |\be_{\ell}|.$
\end{cor}

\begin{exa}
     Let $K = 11a362$ as in Examples \ref{KhFormulaEx} and \ref{JonesFormulaEx}. From the Khovanov homology in Table \ref{exampletable1}, we see that $\rk Kh^{-10, -30}(K) = 2 = \binom{1}{1}+1, \rk Kh^{-9, -28}(K) = 1 = \binom{2}{2}+\binom{2-2}{0}-1, \rk Kh^{-8, -26}(K) = 3 = \binom{3}{3}+\binom{2-1}{1}+1$, and 
         $\rk Kh^{-7, -24}(K) = 2 = \binom{4}{4}+\binom{2}{2}+\binom{2-2}{0}-1$.
     
    These ranks agree with Equation \ref{KhUpperEq} for the values $b = 2$, $\delta = 1$, and $0 < i \le 4$. However, there is no agreement for $i=5$, since
    \begin{equation*}
          \rk Kh^{-6, -22}(K) = 3
    \end{equation*}
    but
    \begin{equation*}
    	\left(\displaystyle\sum_{\substack{r \ge 0,\\ 0\le k = 5-2r \le 5}} \binom{2-2+k}{k} \right) +(-1)^{5+1} = \binom{5}{5}+\binom{3}{3}+\binom{2-1}{1}+1 = 4.
    \end{equation*}
    Using Corollary \ref{KhUpper} we conclude that $M_K = 4+2 = 6$ is an upper bound for $\gr(K)$.
    
    Since $K$ is Khovanov thin, we can obtain the same upper bound for $\gr(K)$ using the Jones coefficients and Corollary \ref{JonesUpper}. From Example \ref{JonesFormulaEx} we have
    \begin{equation*}
        \be_0 = -1, \be_1= 2, \be_2 = -3, \be_3 = 4, \be_4 = -5
    \end{equation*}
    These coefficients alternate in sign and their absolute values satisfy the formula $\binom{b+i-1}{i}$ for $b=2$, $0\leq i\leq 4$. From Corollary \ref{JonesUpper} we derive the upper bound of $M_J = 4+2 = 6$.
   
   Suppose $D$ is a diagram of $K$ such that $G_+(D)$ realizes the maximum girth of 6. Using Corollary \ref{Khgraphproperty}, we determine $G_+(D)$ must be bipartite since $\rk Kh^{-11, -30}(K) = 1.$ The cyclomatic number of $G_+(D)$ must be
   \begin{equation*}
        p_1 = \rk Kh^{-10, 30}(L)-\rk Kh^{-11, 30}(L)+1 = 2
   \end{equation*}
  while the number of 6-cycles in $G_+(D)$ must be
	     \begin{align*}
	           n_{6} &= \left(\sum_{\substack{r  \ge 0,\\ 0 \le k = 5-2r \le 5}} \binom{2-2+k}{k} \right) +(-1)^{6}\rk Kh^{-11, -30}(K) - \rk Kh^{-6, -22}(K)\\
	           &= 3 + 1 - 3 = 1.
	     \end{align*}
    Together, these statements imply that $G_+(D)$ must be a bipartite graph containing exactly 2 cycles: one cycle of length 6 and another cycle of length $n$ where $n>6$ is even.
     
\end{exa}

The following example demonstrates that inequalities \eqref{UpKhov} and \eqref{UpJon} may be strict; i.e., there may be no diagram for a link that realizes either of these upper bounds for girth.

\begin{exa} \label{InequalityEx}
The knot $K=12n821$ is both non-alternating and homologically thin, with  Jones polynomial $q^{-5}- 2q^{-4}+3q^{-3}- 4q^{-2}+5q^{-1}- 5+5q+4q^{2}- 3q^3+2q^4 + q^5$. By Corollary \ref{JonesUpper} applied to the first 5 coefficients, we find an upper bound $M_J = 6$ for the girth of $K$. Similarly, we can apply Corollary \ref{KhUpper} to ranks on the main diagonal of $Kh(L)$ to obtain the same upper bound for girth, $M_K = 6$. However, we can show that no diagram of $K$ exists which achieves this upper bound.\\

Suppose that $K$ has a diagram $D$ such that $G_+(D)$ has girth greater than 2. Then this diagram is both plus-adequate and non-alternating, so $G$ must have a cut-vertex \cite{Stoimenow1}. By Corollary \ref{Khgraphproperty}, $p_1(G_+(D)) = 2$. Since $G_+(D)$ has no loops or multiple edges, it must be a vertex join of two cycles, $P_n*P_m$. Any knot diagram with all-positive graph $P_n*P_m$ is a diagram of an alternating knot: either a connected sum of torus links, or a rational knot. But $K$ has no alternating diagram, so the girth of $12n821$ must be less than or equal to 2.\\

In addition, the same argument can be applied to show that the girth of any all-negative state graph for this knot is less than or equal to 2.
\end{exa}

The 2nd coefficient of the Jones polynomial, which captures the cyclomatic number of $G_+(D)$, uniquely determines the first $\gr(L)-1$ coefficients. In a similar fashion, the first two homological gradings of $Kh(L)$ determine the first $\gr(L)$ gradings. This leads to a somewhat surprising result.

\begin{thm} \label{girthPossibleKhovanov}
	 Let $L$ be a non-trivial link. If $D$ is a diagram of $L$ such that $\ell(G_+(D)) < \gr(L)$, then  $\ell(G_+(D)) = 1$ or  $\ell(G_+(D)) = 2$.
\end{thm}

\begin{proof}
	The result holds for $1 \le \gr(L) \le 3$,  since no non-trivial diagram can have $\ell(G_+(D)) = 0$. For $\gr(L)=M>3$, there exists some diagram $D_{max}$ such that $G_{max} = G_+(D_{max})$ has girth $M$. Suppose that there exists another diagram $D$ of $L$ such that $G = G_+(D)$ has girth $h$, with $2 < h < M$. Applying Theorem \ref{rankKhovanovgirth} to diagram $D_{max}$ for $i = h-1$, we find that the rank of the Khovanov homology group $Kh^{(h-1)-c_-(D_{max}),N(D_{max})+2(h-1)}(L)$ is equal to
	\begin{equation} \label{Compare1}
	\left(\displaystyle\sum_{\substack{r  \ge 0,\\ 0 \le k = (h-1)-2r \le (h-1)}} \binom{p_1(G_{max})-2+k}{k} \right)-n_{h}(G_{max}) +(-1)^{h} \delta^b(G_{max})
	\end{equation} where $N(D_{max}) = -|s_+(D_{max})|+c_{+}(D_{max})-2c_{-}(D_{max})$ and $\delta^b(G_{max})$ is 1 if $G_{max}$ is bipartite, 0 otherwise.

	Now we also apply Theorem \ref{rankKhovanovgirth} to diagram $D$ for $i = h-1$. Since $D_{max}$ and $D$ are both plus-adequate, $c_-(D_{max}) = c_-(D)$. In addition, $-|s_+(D_{max})|+c_+(D_{max}) =  -|s_+(D)|+c_+(D)$ and thus $N(D_{max}) = N(D)$ (see \cite{Lick}). Thus Equation \ref{Compare2} describes the rank of the same group in Khovanov homology as in Equation \ref{Compare1}:
	\begin{equation} \label{Compare2}
	 \left(\displaystyle\sum_{\substack{r  \ge 0,\\ 0 \le k = (h-1)-2r \le (h-1)}} \binom{p_1(G)-2+k}{k} \right)-n_{h}(G) +(-1)^{h} \delta^b(G)
	\end{equation} with $N(D)$ and $\delta^b(G)$  defined as above.

	Since  $\ell(G_{max}) = M > 2$ and  $\ell(G) = h > 2$, we can use Corollary \ref{Khgraphproperty} to calculate $p_1(G_{max}) = p_1(G)$ and $\delta^b(G_{max}) = \delta^b(G)$. Setting Equation \ref{Compare1} equal to Equation \ref{Compare2}, we see that $n_h(G_{max})$ must be equal to $n_h(G)$. However, $n_h(G_{max}) = 0$ by assumption (because $h<M$) while $n_h(G) > 0$ (because $G$ must contain at least one cycle of length $h$). Thus we have a contradiction and no such diagram $D$ may exist.
\end{proof}

\begin{cor}\label{girth3more}
If $L$ has a diagram $D$ such that $ \ell(G_+(D)) = p \ge 3$, then $\gr(L) = p$.
\end{cor}

\begin{exa}[Alternating pretzel links]
	Let $L$ be an alternating pretzel link with twist parameters $(-a_1, -a_2, \ldots, -a_n)$ where $a_i > 1$ for all $i$. The girth of the graph obtained from the standard diagram is $\min\{a_i+a_j ~|~1 \le i \neq j \le n \}$. Since this number is at least 4, it is equal to $\gr(L)$ by Corollary \ref{girth3more}.
\end{exa}

\begin{exa}[3-braids]
Suppose $L$ is the closure of a negative 3-braid $\gamma = \si_{i_1}^{a_1}\si_{i_2}^{a_2}\ldots\si_{i_k}^{a_k}$, $a_j \le -1$. If $i_j \in \{1,2\}$ and $i_j \neq i_{j+1}$ for all $j$, then $\ell(G_+(\gamma)) \ge 3$ \cite[Proposition 5.1]{PS} and so $\gr(L) = \ell(G_+(\gamma))$ by Corollary \ref{girth3more}.
\end{exa}

\subsection{Girth and related knot invariants}

It turns out that girth behaves well under connected sum. Recall that the connected sum of two oriented knots $K_1, K_2$ is well-defined for any choice of planar diagrams for these two knots.

\begin{thm} \label{connectLowerBound}
The girth of a connect sum $ K_1 \# K_2$ of two knots $K_1, K_2$ is equal to the minimum of the girths of these knots: 
$	    \gr(K_1 \# K_2) = \min\{\gr(K_1), \gr(K_2)\}.$
\end{thm}

\begin{proof}
	First we show that $ \min\{\gr(K_1), \gr(K_2)\} \leq \gr(K_1 \# K_2) $. Let $D_1$ be a diagram of $K_1$, $D_2$ be a diagram of $K_2$ such that $\ell(G_+(D_1)) = \gr(K_1)$ and $\ell(G_+(D_2)) = \gr(K_2)$. When we perform the connected sum operation on $D_1$ and $D_2$, the all-positive state graph of the new diagram consists of $G_+(D_1)$ and $G_+(D_2)$ joined at a single vertex, with girth $\min\{\gr(K_1), \gr(K_2)\}$. This gives a lower bound for the girth of $K_1 \# K_2$.
	
	Now we use Corollary \ref{JonesUpper} to prove that $\gr(K_1 \# K_2) \le \min\{\gr(K_1), \gr(K_2)\}$ and thus show equality. Recall that $J_{K_1 \# K_2} = J_{K_1}J_{K_2}$ \cite{Lick}. Let $g_1 = \gr(K_1)$, $g_2 = \gr(K_2)$ and assume without loss of generality that $g_1 \le g_2$. Then the tail of $J_{K_1}$ has the form:
	\begin{align} \label{Tail1}
	     J_{K_1}(q) &= (-1)^{s_1}\left(P_0q^{C_1} - P_1q^{C_1+2} + \ldots +(-1)^{g_1-1} P_{g_1-1}q^{C_1+2(g_1-1)} + \ldots\right)\\
	     &= (-1)^{s_1}\left(\sum_{i=0}^{g_1-1}(-1)^i P_i q^{C_1+2i} + \ldots\right)
	\end{align}
	where $P_i = \binom{b_1-1+i}{i}$ for $0 \le i \le g_1-2$, $P_{g_1-1} = \binom{b_1-1+(g_1-1)}{g_1-1} - n_{g_1}$ (Theorem \ref{Jones1}) and $s_1$, $C_1, b_1 > 0, n_{g_1}>0$ all depend on $K_1$.
	
	Similarly, the tail of the Jones polynomial $J_{K_2}$ has the form:
	\begin{equation} \label{Tail2}
	   J_{K_2}(q) = (-1)^{s_2}\left(\sum_{i=0}^{g_2-1}(-1)^i Q_i q^{C_2+2i} + \ldots\right)
	\end{equation}
		where $Q_i = \binom{b_2-1+i}{i}$ for $0 \le i \le g_2-2$, $Q_{g_2-1} = \binom{b_2-1+(g_2-1)}{g_2-1} - n_{g_2}$, and $s_2$, $C_2, b_2 > 0, n_{g_2}>0$ all depend on $K_2$.
		
		Theorem \ref{Jones1} describes the first $g_1$ coefficients of $J_{K_1}$ and the first $g_2$ coefficients of $J_{K_2}$. The tail coefficients of $J_{K_1 \# K_2} = J_{K_1}J_{K_2}$ result from combinations of the coefficients in Equations \ref{Tail1} and \ref{Tail2}. We write the product as:
	\begin{equation} \label{Tail3}
	     J_{K_1 \# K_2}(q) = (-1)^{s_1+s_2}\left(R_0q^{C_1+C_2} - R_1q^{C_1+C_2+2} + \ldots + (-1)^{g_1-1} R_{g_1-1}q^{C_1+C_2+2(g_1-1)} +  \ldots\right)
	\end{equation}
	where $R_i = \sum_{n=0}^i P_nQ_{i-n}$ for $0 \le i \le g_1-1$. Recalling our assumption that $g_1 \le g_2$, observe that we cannot say anything about the coefficients that follow $R_{g_1-1}$ because we only have $P_0$ through $P_{g_1-1}$ for the first polynomial.
	
	For $0 \le i \le g_1-2$, we compute the following, using a modified form of the Chu-Vandermonde identity (see \cite[Table 3]{Gould1}):
	\begin{equation} \label{sumcoeff}
	    R_i = \sum_{n=0}^i P_nQ_{i-n} = \sum_{n=0}^i \binom{b_1-1+n}{n} \binom{b_2-1+(i-n)}{i-n} = \binom{(b_1+b_2)-1+i}{i}
	\end{equation}
    while on the other hand for $i = g_1-1$:
    \begin{align} \label{sumcoeffLast}
     R_{g_1-1} &= \sum_{n=0}^{g_1-1} P_nQ_{(g_1-1)-n} = \sum_{n=0}^{g_1-2} P_nQ_{(g_1-1)-n} + P_{g_1-1}Q_0
	\end{align}
	Since $P_{g_1-1} = \binom{b_1-1+(g_1-1)}{g_1-1}-n_{g_1}$ with $n_{g_1} > 0$, $R_{g_1-1}$ does not agree with the formula in Equation \ref{sumcoeff} that describes the coefficients from $i=0$ to $i=g_1-2$. Thus the sequence of coefficients $R_i$ for $J_{K_1 \# K_2}$ along with the alternating signs in Equation \ref{Tail3} satisfy the conditions of Corollary \ref{JonesUpper} for $0 \le i \le g_1-2$. Hence the upper bound for $\gr(K_1\#K_2)$ given by the tail of $J_{K_1 \# K_2}$ is $M_J = (g_1 - 2) + 2 = g_1 = \min\{\gr(K_1), \gr(K_2)\}$.
\end{proof}

For a reduced knot diagram $D$, the knot signature and numbers of crossings give upper bounds on the girths of $G_+(D)$ and $G_-(D)$, the graphs related to the all-positive and all-negative Kauffman states $s_+(D)$ and $s_-(D)$. Recall that $\si(L)$ is a link invariant given by the signature of the Seifert matrix obtained from any diagram of $L$. We denote the number of crossings in a diagram by $c(D)$ and the numbers of positive and negative crossings by $c_+(D), c_-(D)$ respectively, using the crossing conventions from \cite{Lick}.

\begin{thm} \label{SignatureBound}
Let $K$ be a non-trivial knot with an oriented, reduced diagram $D$. Then $$\ell(G_+(D)) \le \displaystyle\frac{2c(D)}{c_-(D)-\si(K)+1}.$$
\end{thm}

\begin{proof}
The graph $G_+(D)$ is planar and connected, so the girth $\ell(G_+(D))$ is related to the numbers of edges and vertices by the following inequality: $E \le \displaystyle\frac{\ell(G_+(D))}{\ell(G_+(D))-2}(v-2)$ (see e.g. \cite{Diestel}). Since $D$ is reduced, we may assume that $\ell(G_+(D)) \ge 2$ and rearrange the inequality as $\ell(G_+(D)) \le \displaystyle\frac{2E}{E-v+2}$. The number of edges $E$ is the number of crossings $c(D)$, and the number of vertices is the number of connected components $s_+$ in the all-positive smoothing of $D$. Using the inequality $\sigma(K) \le s_+(D)-c_+(D)-1$ \cite{DasLow2} we obtain the result:
\begin{equation}
    \ell(G_+(D)) \le \displaystyle\frac{2E}{E-v+2} = \displaystyle\frac{2c(D)}{c(D)-s_+(D)+2} \le \displaystyle\frac{2c(D)}{c_-(D)-\si(K)+1} \qedhere
\end{equation}
\end{proof}

A similar proof using $\sigma(K) \le -s_-(D)+c_-(D)+1$ gives an upper bound for the girth of $G_-(D)$.
\begin{cor}\label{SignatureBoundCor} 
Given an oriented, reduced, positive knot diagram $D$ then: $\ell(G_+(D)) \le \dfrac{2}{3}c(D)$. Note that this inequality is sharp.
\end{cor}

\begin{proof}
Since $D$ is positive, $c_-(D) = 0$ and $\si(K) \le -2$ \cite{Przy2}. The standard diagram of the $(-3,-3,-3)$ pretzel knot, which has 9 positive crossings, $\si = -2$, and $G_+(D)$ with girth 6 provides an example where the equality  $\ell(G_+(D)) = \dfrac{2}{3}c(D)$ is achieved.
\end{proof}

\begin{exa} As an application, we observe that a positive alternating knot can never have a diagram $D$ such that $G_+(D)$ is a cycle graph. Such a diagram would have $\ell(G_+(D)) = c(D)$, a contradiction by Corollary \ref{SignatureBoundCor}.
\end{exa}

If we restrict our attention to  alternating links, we can get more specific results about girths of alternating diagrams. In particular, the following theorem states that any two reduced diagrams of a prime alternating link have the same girth.

\begin{thm} \label{AltGirth}
	Let $L$ be a prime alternating link. If $D$, $D'$ are two reduced alternating diagrams of $L$, then $G_+(D)$, $G_+(D')$ have the same girth.
\end{thm}

\begin{proof}
	The Tait flyping conjecture states that any two reduced alternating diagrams of $L$ are related by flypes \cite{MT1}. Flypes may be expressed as a series of mutations, which induce Whitney flips on the corresponding graph \cite{Greene2}. Thus $G_+(D)$ and $G_+(D')$ are 2-isomorphic graphs. Girth is an invariant of the cycle matroid \cite{Oxley}, so $G_+(D)$ and $G_+(D')$ have the same girth.
\end{proof}

\clearpage
\bibliographystyle{abbrv}
\bibliography{bibgirth1}{}

\end{document}